\documentclass[12pt]{amsart}
\DeclareMathOperator{\vol}{vol}

\usepackage{pdfsync}
\usepackage[active]{srcltx}
\usepackage[english]{babel}
\usepackage{amscd}
\usepackage{amssymb}
\usepackage{amsthm}
\usepackage{amsmath}
\usepackage[latin2]{inputenc}
\usepackage{t1enc}
\usepackage{graphicx}
\usepackage{comment}
\usepackage{enumerate}
\usepackage{color}
\usepackage{hyperref}

\usepackage{psfrag}
\usepackage{epsfig}

\usepackage{dsfont}
\usepackage{enumerate}
\usepackage{bbm}
\usepackage{soul}

\definecolor{trp}{rgb}{1,1,1}

\definecolor{red}{rgb}{1,0,.2}

\definecolor{blue}{rgb}{0,0,1}

\definecolor{rgrey}{rgb}{.8,0.4,.4}  % faint
\definecolor{grey}{rgb}{.13,.13,.13}  % almost black

\definecolor{green}{rgb}{0.0,0.4,0.2}

\newtheorem{theorem}{Theorem}
\theoremstyle{plain}

\newtheorem{corollary}[theorem]{Corollary}

\newtheorem{definition}[theorem]{Definition}

\newtheorem{fact}[theorem]{Fact}
\newtheorem{lemma}[theorem]{Lemma}

\newtheorem{proposition}[theorem]{Proposition}
\newtheorem{remark}[theorem]{Remark}

\numberwithin{equation}{section}

\usepackage[usenames,dvipsnames,svgnames]{xcolor}

\newcommand*{\e}[1]{\text{e}^{#1}}
\begin{document}

\parindent0pt

\title[Fractal percolations]
{The dimension of projections of fractal percolations}

\author{Micha\l\ Rams}
\address{Micha\l\ Rams, Institute of Mathematics, Polish Academy of Sciences, ul. \'Sniadeckich 8, 00-956 Warsaw, Poland
\tt{rams@impan.gov.pl}}

\author{K\'{a}roly Simon}
\address{K\'{a}roly Simon, Institute of Mathematics, Technical
University of Budapest, H-1529 B.O.box 91, Hungary
\tt{simonk@math.bme.hu}}

 \thanks{
\indent Rams was partially supported by the MNiSW grant N201 607640 (Poland).
 The research of Simon was supported by OTKA Foundation
\# K 104745}

\begin{abstract}\emph{Fractal percolation} or \emph{Mandelbrot percolation} is
one of the most well studied families of random fractals.
 In this paper we study some of the  geometric measure theoretical properties (dimension of projections and structure of slices)  of these random sets.
Although random, the geometry of those sets is quite regular.
Our results imply that, denoting by $E\subset \mathbb{R}^2$ a typical realization of the fractal percolation on the plane,
\begin{itemize}
  \item If $\dim_{\rm H}E<1$ then for \textbf{all }lines $\ell$ the orthogonal projection $E_\ell$ of $E$ to $\ell$ has the  same Hausdorff dimension as $E$,
  \item If $\dim_{\rm H}E>1$ then for  any smooth real valued  function $f$ which is strictly increasing in both coordinates, the image $f(E)$ contains an interval.
\end{itemize}
The second statement is quite interesting considering the fact that $E$ is almost surely a Cantor set (a {\it random dust}) for a large part of the parameter domain, see \cite{Chayes1988}.
Finally, we solve a related problem about the existence of an interval in the algebraic sum  of  $d\geq 2$ one-dimensional fractal percolations.

\end{abstract}
%\date{\today}

\maketitle

\medskip

\section{introduction}
To model turbulence,
Mandelbrot \cite{Mandelbrot1974}, \cite{Mandelbrot1983} introduced a  family of statistically self-similar random sets $E$  which is now called  fractal percolation or Mandelbrot percolation.
This is a two-parameter $(M,p)$ family of random sets in $\mathbb{R}^d$, where $M\geq 2$ is an integer and $0<p<1$ is a probability. The inductive construction of $E$ is as follows.
The (closed) unit cube of $\mathbb{R}^d$ is divided into
$M^d$ congruent cubes. Each of them are retained with probability $p$ and discarded with probability $1-p$. In the retained cubes we repeat this division and retaining/discarding process independently of everything at infinitum or until there are no retained cubes left. The random set  $E$ that remains after infinitely many steps (formally: the intersection of the unions of retained cubes on all stages of the construction) is the fractal percolation set.
See Section \ref{154} for a more detailed description.

The number of retained cubes of level $n$  forms a branching process with   offspring distribution $\texttt{Binomial}(M^d,p)$, which will be our standing assumption.
So, $E\ne\emptyset $ with positive probability iff  $p>1/M^d$.
%In Mandelbrot's model, turbulence occurs on the plane
An interesting phenomenon appears in $d \ge 2$
when the opposite walls of the unit square are connected in $E$, it is called percolation.
Chayes, Chayes and Durrett \cite{Chayes1988} proved that
this happens with positive probability when $p>p_{\mathrm{crit}}$ a critical probability. We do not know the precise value  of $p_{\mathrm{crit}}$ but it was proved in \cite{Chayes1988} that
$p_{\mathrm{crit}}<1$. Further, it was also proved in \cite{Chayes1988} that
for $p<p_{\mathrm{crit}}$ the random set $E$ is totally disconnected (a dust)
almost surely conditioned on $E\ne\emptyset $.

We consider first the fractal percolations on the plane: $d=2$. We study the projections and slices of $E$. In particular, we point out (Theorem \ref{u1}) that for \emph{\textbf{all}} lines $\ell$ we have
\begin{equation}\label{z3}
\dim_{\rm H}E_\ell=\min\left\{1,\dim_{\rm H}E\right\},
\end{equation}
where $E_\ell$ is the orthogonal projection of the set $E$ to the line $\ell$.
 So, for every set $E$ having Hausdorff dimension smaller than one,  the dimension is preserved by \textbf{all} orthogonal projections.
We remark that  the well known theorem of Marstrand \cite{Mattila1999} guarantees the same
only for lines $\ell$ in Lebesgue almost all directions.

\medskip

In fact we prove much more than \eqref{z3}. Namely, in Corollary \ref{u66} we point out that whenever $\dim_{\rm H}E<1$ (that is
$1/M^2<p<1/M$), for almost all realizations of $E$, \textbf{all} lines
intersect at most $cn$ (the constant $c$ may depend on the realization) level $n$ squares that are retained (among the exponentially many retained level $n$-squares).
This observation  is the main contribution of our paper to this field. It gives much more precise information than \eqref{z3}, we are also able to apply this result to the algebraic sums of independent fractal percolations on the line (see Theorem \ref{thm:prod}).

\medskip

If we are still on the plane and $1/M<p<p_{\mathrm{crit}}$
then $E$ is random dust with
$\dim_{\rm H}E>1$ (see \eqref{155}).
It follows from our earlier result  \cite{Rams} that
in this case all kind of projections of $E$ contain some intervals.
In this paper we verify that any smooth image $f(E)$ of $E$ by a componentwise strictly increasing function $f$ (with non-zero partial derivatives) contains an interval. This shows that although $E$ is  a random dust, some  of its geometric measure theoretical properties are  rather regular.

Finally, we investigate the existence of some intervals in the algebraic sum of random Cantor sets. This theme of research arose naturally in relation with the hyperbolic behavior  of  some one-parameter families of diffeomorphisms, see \cite{Palis1987}
for a comprehensive account.
It was proved in \cite{Dekking2008} that
the algebraic sum  of two fractal percolations (which is actually the $45^{\circ}$ degree projection of the product set)
contains an interval, if and only if the sum of their Hausdorff dimension is greater than one. We extend this result to any dimension $d\geq 3$. The major difficulty in this generalization
is to handle the problem caused by
the presence of much more dependence in between the cubes of the level $n$ approximation of the product of $d\geq 3$ fractal percolations than that of in the case when $d=2$.

\section{Notation}\label{154}
\subsection{The $d$-dimensional fractal percolation with parameters $M,p$}\label{u87}

The intuitive definition  provided below is given in $\mathbb{R}^d$ for an arbitrary $d\geq 1$  (for a more formal definition see \cite[Section 2.2]{Dekking2009}). Fix a natural number $M\geq 2$ and a probability $0<p<1$. Throughout the construction we define a random, nested sequence $E_n$ which is the union of some randomly chosen level-$n$ cubes.
These are the $M$-adic cubes, that is coordinate-hyperplane  parallel cubes of side length $M^{-n}$ with centers chosen from
$$
\pmb{\mathcal{N}_n}:=
\left\{\mathbf{x}=(x_1,\dots ,x_d):x_i=\left(k_i+\frac{1}{2}\right)M^{-n},\ 0\leq k_i\leq M^n-1\right\}.
$$
We denote
the level-$n$ cube with center $\mathbf{x}\in \mathcal{N}_n$ by $K_n(\mathbf{x})$.
$$
K_n(\mathbf{x})=
\mathbf{x}+\left[-\frac{1}{2M^n},\frac{1}{2M^n}\right]^d.
$$
We will sometimes identify the level-$n$ cubes with their centers. We write $\mathcal{N}_n$
for the collections of level $n$ cubes.

The construction of the fractal percolation is as follows.
We start with the unit cube $K=[0,1]^d$.
For every $\mathbf{x}\in \mathcal{N}_1$ we retain the cube $K_1(\mathbf{x})$  with probability $p$ and  we discard it with probability $1-p$, independently.
The union of  cubes
retained is denoted $E_1$. For every retained cube $K_1(\mathbf{x})$ we consider all cubes $K_2(\mathbf{y})\subset K_1(\mathbf{x})$ and each of those is retained with probability $p$, discarded with probability $1-p$ independently.   The union of retained level-$2$ cubes is denoted by $E_2$.
We continue this process ad infinitum to obtain $E_n$ for every $n$. In each step the retaining/discarding of every cube are independent events.  Clearly, $E_n\subset E_{n-1}$.
For $n\geq 1$ set
\begin{equation}\label{v11}
 \pmb{\mathcal{E}_n}:=
\left\{\mathbf{x}\in \mathcal{N}_n:K_n(\mathbf{x})\mbox{ is retained}\right\} .
\end{equation}

The \emph{$d$-dimensional fractal percolation with parameters $M,p$} is the random set $E=E(d,M,p)$
$$
\pmb{E}:=\bigcap\limits_{n=1}^{\infty } E_n.
$$
We call $E_n$  the \emph{$n$-th approximation of the fractal percolation}.
The corresponding probability space $\pmb{(\Omega ,\mathcal{F},\mathbb{P})}$ can be described in terms of infinite $M^d$-ary labeled trees  (see e.g. \cite[Section 2]{Dekking2009} for the details.) Further, we write $\pmb{\mathcal{F}_n}\subset \mathcal{F}$ for the $\sigma $-algebra generated by the selected level-$n$ cubes.

\begin{remark}\label{u84}
A very important feature of the construction is that
$$
E \mbox{ is \textbf{statistically self-similar} with completely\textbf{ independent } cylinders. }
$$
That is
\begin{description}
  \item[(a)] For every $n\geq 1$ and $\mathbf{x}\in \mathcal{E}_n$, an appropriately re-scaled copy  of the random set $E\cap K_n(\mathbf{x})$ has the same distribution as $E$ itself.
  \item[(b)] The sets $\left\{E\cap K_n(\mathbf{x})\right\}_{\mathbf{x}\in \mathcal{E}_n}$ are independent.
\end{description}
\end{remark}

As we have already mentioned, $\left\{\#\mathcal{E}_n\right\}$
is a branching process with offspring distribution $\texttt{Binomial}(M^d,p)$.
 Hence
\begin{equation}\label{163}
\mathbb{P}\left( E\ne \emptyset\right)>0
\mbox{ if and only if } p>\frac{1}{M^d}.
\end{equation}
 Falconer \cite{Falconer1986} and  Mauldin, Williams \cite{Mauldin1986} proved that
\begin{equation}\label{155}
 E\ne \emptyset \mbox{ implies that }
  \dim_{\rm H}(E)=\dim_{\rm B}(E)=\frac{\log p\cdot M^d}{\log M}\mbox{ a.s. }
\end{equation}
In both parts of the paper the intersection of $E$ with hyperplanes play the most important role so we introduce a notation related to it.
Let $H$ be a hyperplane in $\mathbb{R}^d$.  Set
$$
\pmb{\mathcal{E}_{n}(H)}:=\left\{\mathbf{x}\in \mathcal{E}_n:\texttt{int}(K_n(\mathbf{x}))\cap H\ne\emptyset \right\}.
$$
The $d-1$ dimensional Lebesgue measure of $E_{n}\cap {H}$ is
\begin{equation}\label{u85}
  \pmb{L_{n}(H)}:=\mathcal{L}{\rm eb}_{d-1}(E_{n}\cap {H})=
\sum\limits_{\mathbf{x}\in \mathcal{E}_{n}(H)}
\mathcal{L}{\rm eb}_{d-1}\left(
K_n(\mathbf{x})\cap H\right).
\end{equation}

\section{The slices and orthogonal projections of $E$ on the plane. The case of small $E$. }

In this section $d=2$.
The main result of the first part of the paper is
\begin{theorem}\label{u1} For every $0<p\leq 1$ for almost every realization of $E$ (conditioned on $E\ne\emptyset $),
for \textbf{all} straight lines $\ell$ which are not parallel to the coordinate axes
 we have:
\begin{equation}\label{u2}
 \dim_{\rm H}(E_\ell)=\min\left\{1,\dim_{\rm H}(E)\right\}.
\end{equation}
\end{theorem}
We remark that in the special case when  $\ell$ is parallel to one of the coordinate axis, \eqref{u2} was verified by
Falconer \cite{Falconer1989} (in a special case) and Dekking, Meester \cite{Dekking1990} in full generality. (See also \cite{Dekking1988} for the same result for box dimension.)
Combining this with our theorem above we can state:
\begin{corollary}\label{v14}For every $0<p\leq 1$
for almost every realization of $E$ (conditioned on $E\ne\emptyset $),
for \textbf{all} straight lines $\ell$
 we have:
\begin{equation}\label{v15}
 \dim_{\rm H}(E_\ell)=\min\left\{1,\dim_{\rm H}(E)\right\}.
\end{equation}
\end{corollary}
In \cite{Rams} we proved that whenever $Mp>1$ (that is conditioned on $E\ne\emptyset $, a.s. $\dim_{\rm H}(E)>1$)
for almost all realization $\omega $ leading to a nonempty $E$, for all straight lines $\ell$, the orthogonal projection $E_\ell(\omega )$ contains some intervals. This implies that the assertion of our  Theorem \ref{u1} holds whenever $Mp>1$.
Using this and (\ref{163}) without loss of generality
in the rest of the section we may always assume that

\textbf{Principal Assumption for this Section:}
\begin{equation}\label{u3}
 M^{-2}<p\leq M^{-1}.
\end{equation}

\subsection{Projection $\mathrm{proj}^\alpha$}
We define the argument $\pmb{\mathrm{Arg}(\ell)}\in \left[0,\pi/2 \right)$ of a line as the oriented angle it makes with the $x$-axis.
Instead of considering the orthogonal projection of $E$ to a line $\ell$ in direction $\theta$ we will consider the linear projection (in the same direction $\alpha=\theta \pm \pi/2$) to one of diagonals of $K$. This replacement does not change the Hausdorff dimension of the projection.

More formally, let $\alpha \in (0,\pi)\setminus \{\pi/2\}$ (we are not interested in the horizontal and vertical projections).
If $\alpha \in \left(0,\frac \pi 2\right)$ then $\pmb{\Delta ^{\alpha }}$ denotes the
decreasing diagonal of $K$ (the diagonal connecting points $(0,1)$ and
$(1,0)$). If $\alpha \in \left(\frac{\pi }{2},\pi \right)$ then
$\Delta ^\alpha $ is the increasing diagonal of $K$. For an
$\alpha \in \left(0,\pi \right)\setminus \{\pi/2\}$ we write $\pmb{\mathrm{proj}^\alpha}:K\to \Delta^ \alpha $ for the
angle $\alpha $ projection to the diagonal $\Delta^ \alpha $ in
$K$.
Without loss of generality  we may confine ourselves to the angle
\begin{equation}\label{u6}
  0<\alpha <\pi /2.
\end{equation}
projections to the  decreasing diagonal which we denote by $\Delta $.

\subsection{The slices}

In this and the following sections we study the length of the intersection of any lines with the  $n$-th approximation of fractal percolation. We formulate the results only for lines which are neither horizontal nor vertical because this is the only case that we are going to apply. However, the assertion of Theorem \ref{u36} also holds for horizontal and vertical lines. Since the proof is not very much different we omit it.

Consider the family of all lines with argument between 0 and $\pi/2$ having non-empty intersection with $\mathrm{int}(\Delta) $. The unit square $K$ cuts out a line segment from each of these lines. Let $\pmb{\mathfrak{L}}$ be the set of all line segments obtained in this way.
The sets of the form $E\cap \ell, \ell\in \mathcal{L}$
are the slices of $E$. The segment  $\ell\in \mathcal{L}$  which has argument $\alpha $ and intersects $\Delta $ at the point $z$  is denoted by $\pmb{\ell^{\alpha} (z) }$.

We will study the length of the slices of the  level-$n$
approximation $E_n$:
\begin{equation}\label{u31}
\pmb{L_{n}(\ell)}:=\left|E_n\cap \ell\right|,\quad \ell\in \mathfrak{L}.
\end{equation}
It is immediate  from  the construction of the fractal percolation that
for every $\ell\in \mathfrak{L}$, $n\geq 1$,
\begin{equation}\label{u90}
\forall \mathbf{x}\in \mathcal{N}_{n-1},\   \mathbb{E}\left[
|E_n\cap \ell\cap K_{n-1}(\mathbf{x})|\Big|\mathbf{x}\in \mathcal{E}_{n-1}
\right]
=
  p|\ell\cap K_{n-1}(\mathbf{x})|.
\end{equation}
Clearly, $\mathfrak{L}$ can be presented as a countable union of families of lines segments
$\mathfrak{L}^\theta $ whose angles $\mathrm{Arg}(\ell)$ are $\theta $-separated from both $0$ and $\pi /2$:
$$
\pmb{\mathfrak{L}^\theta} :=\left\{
\ell\in \mathfrak{L}:\min\left\{\mathrm{Arg}(\ell),\frac{\pi }{2}-\mathrm{Arg}(\ell)\right\}>\theta
\right\}, \quad 0<\theta <\pi /4.
$$
Then $\mathcal{L}=\bigcup _{k=2}^{\infty }\mathcal{L}^{1/k}$.
%%%%%%%%%%%%%%%%%%%%%%%%
 We would like to get an upper bound for $\#\mathcal{E}_n(\ell)$ for an arbitrary $\ell\in \mathfrak{L}^\theta $. To do so, first we give a uniform upper bound for $L_n(\ell)$ for all $\ell\in \mathfrak{L}^\theta$ and then we use the following easy fact:
\begin{fact}\label{u96}
Let $\mathcal{E}_{n}^{\mathrm{bigg}}(\ell):=\left\{\mathbf{x}\in \mathcal{E}_{n}(\ell):
\left|\ell\cap K_n(\mathbf{x})\right|\geq M^{-n}/\sqrt[]{2}
\right\}$, analogously set $\mathcal{E}_{n}^{\mathrm{small}}(\ell):=\left\{\mathbf{x}\in \mathcal{E}_{n}(\ell):
\left|\ell\cap K_n(\mathbf{x})\right|<M^{-n}/\sqrt[]{2}
\right\}$. Let $\ell^u,\ell^l$ be  lines which are parallel to $\ell$, their distance from $\ell$ is  in $\left(M^{-n}/2\ \sqrt[]{2},M^{-n}/\sqrt[]{2}\right)$ and they lie on  opposite sides of $\ell$. Then
\begin{equation}\label{u94}
\mathcal{E}_{n}^{\mathrm{small}}(\ell)\subset
\mathcal{E}_{n}^{\mathrm{bigg}}(\ell^u)\cup
\mathcal{E}_{n}^{\mathrm{bigg}}(\ell^l)
\end{equation}
That is
\begin{equation}\label{u95}
  \#\mathcal{E}_n(\ell)\leq 2M^n\left(L_n(\ell)+L_n(\ell^u)+L_n(\ell^l)\right).
\end{equation}

\end{fact}

%%%%%%%%%%%%%%%

We will need the following $M^{-n}$-dense subset of $\mathfrak{L}^\theta $
\begin{definition}\label{u91}
For every $0<\theta <\pi /4$ pick an arbitrary $ M^{-2n}$-dense set subset
$\Delta _n\subset \Delta $ and also an $M^{-2n}$-dense subset $A_{n}^{\theta }$ of $\left(\theta ,\frac{\pi }{2}-\theta \right)$
such that $\Delta _n\subset \Delta _{n+1}$ and $A_{n}^{\theta }\subset A_{n+1}^{\theta }$. Let
$$
\mathfrak{L}_{n}^{\theta }:=\left\{\ell\in \mathfrak{L}^{\theta }:
z(\ell)\in \Delta _n\mbox{ and } \mathrm{Arg}(\ell)\in A_{n}^{\theta }.
\right\}
$$
\end{definition}
Clearly,
\begin{equation}\label{u92}
  \mathfrak{L}_{n}^{\theta }\subset \mathfrak{L}_{n+1}^{\theta } \mbox{ and }
  \#\mathfrak{L}_{n}^{\theta }=M^{4n}.
\end{equation}
By elementary geometry we get
\begin{fact}\label{u49}
For every $0<\theta <\pi /4$ we can find an $s_\theta $  such that for every $\ell\in \mathfrak{L}^\theta $ we can choose  an $\ell'\in \mathfrak{L}_n^\theta $ satisfying
\begin{equation}\label{u93}
  L_{n-1}(\ell)\leq L_{n-1}(\ell')+s_\theta M^{-(n-1)}.
\end{equation}
\end{fact}
We fix such an $s_\theta $ for every $\theta $.
\subsection{The length of the slices}

In this Section we prove a theorem which says that for almost all realizations, the length of \textbf{all } (non-vertical, non-horizontal) level-$n$ slices of angle $\alpha $ are less than $\mathrm{const}\cdot n M^{-n}$ if $n$ is big enough.

\begin{theorem}\label{u36}There exists a $C_2$ (defined in (\ref{v5})) such that for all $0<\theta  <\frac{\pi }{4}$  the following holds almost surely:

\begin{equation}\label{u65}
\exists N,\ \forall n\geq N,\
\forall \ell\in \mathfrak{L}^\theta ;\quad
L_{n}(\ell) <C_2M^{-n}n .
\end{equation}
Here the threshold $N$ depends on both $\theta $ and the realization.
\end{theorem}
Using that $\mathfrak{L}=\bigcup _{k=2}^{\infty }\mathfrak{L}^{1/k}$ we obtain that
\begin{corollary}
Almost surely, for all
$\ell\in \mathfrak{L} $
\begin{equation}\label{u37}
\exists N,\ \forall n\geq N,\quad  L_{n}(\ell)\leq C_2\cdot M^{-n}n.
\end{equation}

\end{corollary}
Using Fact \ref{u96} we get
\begin{corollary}\label{u66}
For almost all realizations of $E$  we have
\begin{equation}\label{v6}
\forall \theta \in \left(0,\frac{\pi }{4}\right),\
\exists N,\ \forall n\geq N,\ \forall \ell\in \mathfrak{L}^\theta; \quad
\#\mathcal{E}_n(\ell)\leq 6C_2n.
\end{equation}

\end{corollary}

To prove Theorem \ref{u36}  we apply a version of  Azuma-Hoeffding inequality to estimate $L_{n}^{\alpha }(x)$.

\subsection{Large deviation estimate for $L_{n}(\ell) $}\label{u45}
An immediate reformulation of the Azuma-Hoeffding inequality \cite[Theorem 2]{Hoeffding1963} yields:
\begin{theorem}[Hoeffding]\label{u40}
Let $X_1,\dots ,X_m$ be independent bounded random variables with $a_i\leq X_i\leq b_i$, ($i=1,\dots ,m$). Then for any $t>0$
\begin{equation}\label{u41}
  \mathbb{P}\left(X_1+\cdots +X_m-\mathbb{E}\left[X_1+\cdots +X_m\right]\geq t\right)
\leq
\exp\left(\frac{-2t^2}{\sum\limits_{i=1}^{m}(b_i-a_i)^2}\right).
\end{equation}

\end{theorem}
We apply this to prove:
\begin{lemma}\label{u35}
For every $u>1$
 there is a constant $r=r(u)>0$ such that for every
$n\geq 1$, $\ell\in \mathfrak{L}$ and $0<R<|\ell|$,
\begin{equation}\label{u39}
  \mathbb{P}\left(L_{n}(\ell) >pL_{n-1}(\ell)\cdot u| L_{n-1}(\ell)\geq R \right)
<
\exp\left(
-r M^{(n-1)}R
\right)
\end{equation}
\end{lemma}
\begin{proof}
Fix an arbitrary $u>1$, $n$, $\ell\in \mathfrak{L}$ and $0<R< |\ell|$. Let
\begin{equation}\label{u44}
  r=\sqrt[]{2}\cdot (u-1)^2p^2.
\end{equation}
We write $\pmb{\mathfrak{N}_{n-1}(\ell)}$ for the collection of all $\texttt{N}\subset \mathcal{N}_{n-1}$ satisfying
\begin{description}
  \item[(a)] $\forall \mathbf{x}\in \texttt{N}$, we have $ \ell\cap \texttt{int}(K_{n-1}(\mathbf{x}))\ne\emptyset $ and
  \item[(b)] $ \sum\limits_{\mathbf{x}\in \texttt{N}}\left| \ell\cap \texttt{int}(K_{n-1}(\mathbf{x}))\right|\geq R$.
\end{description}
For an $\texttt{N}\in \mathfrak{N}_{n-1}(\ell)$ let
$$
\widetilde{\texttt{N}}\mbox{ be the event that }
\texttt{N}=\left\{\mathbf{x}\in \mathcal{E}_{n-1}:\ \texttt{int} (K_{n-1}(\mathbf{x}))\cap \ell\ne\emptyset \right\}.
$$
Note that
$$
\left\{L_{n-1}(\ell)\geq R\right\}=\bigcup_{\texttt{N}\in \mathfrak{N}_{n-1}(\ell)} \widetilde{\texttt{N}}
$$

with disjoint union. Hence to verify \eqref{u39} it is enough to prove that
\begin{equation}\label{v7}
\forall \texttt{N}\in \mathfrak{N}_{n-1}(\ell),\
  \mathbb{P}\left(
  L_n(\ell)>pL_{n-1}(\ell)\cdot u|\widetilde{\texttt{N}}
  \right)<\exp\left(-rM^{n-1}R\right).
\end{equation}
Fix an arbitrary $\texttt{N}\in \mathfrak{N}_{n-1}(\ell)$
and set
$$
\pmb{\widetilde{\mathbb{P}}}(\cdot):=\mathbb{P}\left(\cdot |\widetilde{\texttt{N}}\right).
$$
Clearly, $L_{n-1}(\ell)$ is deterministic on $\widetilde{\texttt{N}}$. By definition,
\begin{equation}\label{u89}
  L_{n-1}(\ell)=
\sum\limits_{\mathbf{x}\in \texttt{N}}
|\ell\cap K_{n-1}(\mathbf{x})|\geq R.
\end{equation}
With this notation (\ref{v7}) is of the form:
\begin{equation}\label{u42}
 \widetilde{ \mathbb{P}}\left(L_{n}(\ell) >puL_{n-1}(\ell)\right)
<
\exp\left(
-r M^{n-1}R
\right).
\end{equation}
  Let $\left\{X_\mathbf{x}\right\}_{\mathbf{x}\in \texttt{N}}$ be independent random variables on $\widetilde{\texttt{N}}$ with

$X_\mathbf{x}\stackrel{d}{=}\left|\ell \cap E_n\cap K_{n-1}(\mathbf{x})\right|$.
Then by (\ref{u90})
$$
L_{n}(\ell)-p\cdot L_{n-1}(\ell)=
\sum\limits_{\mathbf{x}\in \texttt{N}}
\left(X_\mathbf{x}-\mathbb{E}\left[X_\mathbf{x}\right]\right).
$$
We will apply Theorem \ref{u40} for the random variables $X_\mathbf{x}$.
Using the notation of   Theorem \ref{u40} , observe that
$$a_\mathbf{x}:=0\leq X_\mathbf{x}\leq |\ell\cap K_{n-1}(\mathbf{x})|=:b_\mathbf{x}.$$ The sum on the right hand side of the formulae (\ref{u41}) satisfies
\begin{eqnarray}   \label{u43}
% \nonumber to remove numbering (before each equation)
  \sum\limits_{\mathbf{x}\in \texttt{N}}
  (b_\mathbf{x}-a_\mathbf{x})^2&\leq& \sum\limits_{\mathbf{x}\in \texttt{N}}|\ell\cap K_{n-1}(\mathbf{x})|^2
   \\
 \nonumber  &= &
  2M^{-2(n-1)}
   \sum\limits_{\mathbf{x}\in \texttt{N}}
   \left(
   \frac{M^{n-1}}{\sqrt{2}}\left|\ell\cap K_{n-1}(\mathbf{x})\right|
   \right)^2
  \\
 \nonumber  &\leq &
 \sqrt[]{2}\cdot M^{-(n-1)} \sum\limits_{\mathbf{x}\in \texttt{N}}|\ell\cap K_{n-1}(\mathbf{x})|
 \\
\nonumber   &=& \sqrt[]{2}M^{-(n-1)}L_{n-1}(\ell),
\end{eqnarray}
where in the one but last step we used that all the summands are smaller than or equal to $1$.
Using this and Theorem \ref{u40} for $t=(u-1)pL_{n-1}(\ell) $ on the space $(\widetilde{\texttt{N}},\widetilde{\mathbb{P}})$ we obtain  that
\begin{eqnarray*}
% \nonumber to remove numbering (before each equation)
\widetilde{ \mathbb{P}}\left(L_{n}(\ell) >
pL_{n-1}(\ell)\cdot u \right)  &=& \widetilde{\mathbb{P}}
\left(\sum\limits_{\mathbf{x}\in O}
\left(X_\mathbf{x}-\mathbb{E}\left[X_\mathbf{x}\right]\right)
>t
\right) \\
   &\leq & \exp\left(\frac{-2t^2}{\sum\limits_{i=1}^{m}(b_i-a_i)^2}\right) \\
   &\leq &  \exp\left(
-rM^{n-1} L_{n-1}
\right)\\
&\leq &
\exp\left(-rM^{n-1} R\right),
\end{eqnarray*}
where in the last step we used (\ref{u89}).
 So, (\ref{u42}) holds which implies the assertion of the Lemma.
\end{proof}

We use this Lemma for a re-scaled version of $L_{n}(\ell) $. Namely, let
$$
F_{n}(\ell) :=L_{n}(\ell) \cdot M^{n}.
$$
Then we can reformulate (\ref{u39}) as follows:
\begin{corollary}\label{u53}
For an $u>1$ let $r=r(u)=\sqrt[]{2}\cdot (u-1)^2p^2$. For all $\ell\in \mathfrak{L}$, and
\begin{equation}\label{u56}
  0<R_n\leq M^{n-1}|\ell|,
\end{equation}
 we have
\begin{equation}\label{u46}
  \mathbb{P}\left(F_{n}(\ell) >pMF_{n-1}(\ell)\cdot u|F_{n-1}(\ell)\geq R_n  \right)
<
\exp\left(-rR_n \right).
\end{equation}
\end{corollary}

\subsection{The proof of Theorem \ref{u36} }
Now we prove Theorem \ref{u36} using the large deviation estimate of Corollary \ref{u53}.

Fix an arbitrary $0<\theta  <\frac{\pi }{4}$ for this Section and let $0<\varepsilon <\min\left\{p,\frac{1}{10}\right\}$ such that $Mp(1+\varepsilon )<1$ .
 We will use Corollary \ref{u53} with
\begin{equation}\label{v4}
  u=1+\varepsilon /3\mbox{ that is } r=\sqrt[]{2}p^2\varepsilon ^2/9.
\end{equation}
Set
\begin{equation}\label{u54}
  \pmb{a_n}:=\max\limits_{\ell\in \mathfrak{L}_n^{\theta  }}F_{n-1}(\ell),\
 \pmb{ b_n}:=\frac{8\log M}{r}\cdot n,
\end{equation}
where $\mathcal{L}_{n}^{\theta }$ was defined in Definition \ref{u96}.
The reason for this particular choice of  $b_n$ is to  ensure that
\begin{equation}\label{u97}
  M^{4}\cdot \exp\left(-rb_n/n\right)<1
\end{equation}
which we will need later to apply Borel-Cantelli Lemma.
Clearly,
\begin{equation}\label{u55}
  a_{k+1}\leq M\cdot a_k \mbox{ and } b_{k}< b_{k+1}.
\end{equation}

 Now we prove that
\begin{lemma}\label{u57}
For almost all realizations there exists an $N_0$ (which depends on the realization) such that
\begin{equation}\label{u58}
  \forall n\geq N_0,\ \mbox{ either } a_n\leq b_n \mbox{ or } a_{n+1}\leq \lambda a_n,
\end{equation}
where $\lambda =pM(1+\frac{2\varepsilon }{3})<1$.
\end{lemma}
 \begin{proof}
 We define the events
 $$
 \mathcal{A}_n(\ell):=\left\{
 F_n(\ell)>pMu\cdot F_{n-1}(\ell),\  F_{n-1}(\ell)>b_n
 \right\}
 $$
 and
 $$
 \mathcal{A}_n:=\bigcup _{\ell\in \mathfrak{L}_{n}^{\theta }} \mathcal{A}_n(\ell).
 $$
 Note that $\mathcal{A}_n(\ell)=\emptyset $ for those $\ell\in \mathfrak{L}_{n}^{\theta }$ satisfying $b_{n}>M^{n-1}|\ell|$. Otherwise, we can use
 Corollary \ref{u53} to obtain that
\begin{equation}\label{u98}
\mathbb{P}\left(
F_n(\ell)>pMu\cdot F_{n-1}(\ell),\ F_{n-1}(\ell)\geq b_n
\right)  \leq \exp\left(-rb_n\right).
\end{equation}
Using this and (\ref{u92})  we obtain that
\begin{equation}\label{u59}
\mathbb{P}\left(\mathcal{A}_n\right)\leq M^{4n}\e{-rb_n},
\end{equation}
which is summable by (\ref{u97}). The Borel-Cantelli Lemma yields that for almost all realizations there exists an $N'$ such that for all
$n\geq N'$ we have
\begin{equation}\label{v2}
\forall \ell\in \mathfrak{L}_{n}^{\theta }
\mbox{ either }
F_{n-1}(\ell)\leq b_n
\mbox{ or }
F_{n}(\ell) <pM\left(1+\frac{\varepsilon }{3}\right)F_{n-1}\left(\ell\right).
\end{equation}
Let $n>N'$ and assume that $a_n>b_n$. Choose $\ell'\in \mathfrak{L}_{n}^{\theta }$  such that
$F_{n-1}(\ell')\geq F_{n-1}(\ell)$ for all $\ell\in \mathfrak{L}_{n}^{\theta }$. Then
\begin{equation}\label{u99}
  F_{n-1}(\ell')>b_n.
\end{equation}
Fix an arbitrary $\ell\in \mathfrak{L}_{n}^{\theta }$. If $F_{n-1}(\ell)\leq b_n$ then
\begin{equation}\label{v1}
  F_n(\ell)\leq M^{-1}b_n<M^{-1}F_{n-1}(\ell')<pM\left(1+\frac{\varepsilon }{3}\right)\cdot F_{n-1}(\ell'),
\end{equation}
since $p>M^{-2}$. On the other hand, if $F_{n-1}(\ell)> b_n$ then by (\ref{v2}) and the definition of $\ell'$ we get
\begin{equation}\label{v3}
F_n(\ell)<pM\left(1+\frac{\varepsilon }{3}\right)\cdot F_{n-1}(\ell').
\end{equation}
We choose $N_0\geq N'$ such that for all $n\geq N_0$ we have $\frac{\varepsilon Mp }{3}b_n>s_\theta $. Then by
(\ref{v3}), (\ref{v1}) and (\ref{u99}) we obtain
\begin{eqnarray*}
% \nonumber to remove numbering (before each equation)
 a_{n+1}  &=& \max_{\ell\in \mathfrak{L}_{n+1}^{\theta }}F_n(\ell) \\
   &\leq & \max_{\ell\in \mathfrak{L}_{n}^{\theta }}F_n(\ell)+s_\theta  \\
   &\leq &  pM\left(1+\frac{\varepsilon }{3}\right)\cdot F_{n-1}(\ell')+s_\theta \\
   &<& \underbrace{pM\left(1+\frac{2\varepsilon }{3}\right)}_\lambda \cdot a_n,
\end{eqnarray*}
since we assumed that $a_n>b_n$.
 \end{proof}

\begin{proof}[The proof of Theorem \ref{u36}]
Let $N_0$  and $\lambda $ be as in Lemma \ref{u57}.
First we show that there is an $N_1>N_0$ such that $a_{N_1}\leq b_{N_1}$.
Namely, if $a_{N_0+l}>b_{N_0+l}$ then by Lemma \ref{u57}, $a_{N_0+l+1}\leq \lambda a_{N_0+l}$.
Since $\lambda <1$ and $\left\{b_n\right\}$ is increasing we find a $N_1>N_0$ such that $a_{N_1}< b_{N_1}$. Then for all $k>N_1$ we have
\begin{equation}\label{u61}
  a_k\leq Mb_k.
\end{equation}
Namely, consider the ratio $r_k:=\frac{a_k}{b_{k}}$. If $r_k<1$ (as it happens for $k=N_1$) and $r_{k+1}>1$ then
$r_{k+1}<M$ since $a_{k+1}\leq Ma_k$ and $\left\{b_k\right\}$ is increasing. Then by Lemma \ref{u57} we have $r_{k+i+1}<\lambda r_{k+i}$ as long as $r_{k+i}>1$. Then the same cycle is repeated which completes the proof of (\ref{u61}). Using (\ref{u61}) we obtain that
 almost surely there is an $N_1$ such that for $n\geq N_1$
\begin{equation}\label{u62}
  F_{n}(\ell) \leq \frac{8M\log M}{r}\cdot (n+1), \mbox{ if }\ell\in \mathfrak{L}_{n+1}^{\theta  }.
\end{equation}
Then by  Fact \ref{u49}
\begin{equation}\label{u63}
  \forall  \ell'\in \mathfrak{L}^{\theta  },
\quad
F_{n}(\ell')<\frac{8M\log M}{r}\cdot n+
\left(\frac{8M\log M}{r}
+s_\theta\right)<C_2\cdot n ,
\end{equation}
for
\begin{equation}\label{v5}
  C_2:=\frac{8M\log M}{r}+1
\end{equation}
and $n$ big enough.
(We remind that $r$ was defined in (\ref{v4}).)
\end{proof}

\subsection{The proof of the main result of the Section} Now we prove
Theorem \ref{u1}.
Using the Mass distribution principle \cite{Falconer2003},
Theorem \ref{u1}  follows from the combination of Theorem \ref{u36} and Frostman's Lemma (\cite[Theorem 8.8]{Mattila1999}) with  (\ref{155}). For the convenience of the reader here we cite Frostman's Lemma from
\cite{Mattila1999}.
For a $d\geq 1$ and $B\subset \mathbb{R}^d$ let
$
\mathcal{M}(B)
$ be the set of Radon measures $\mu $ supported by $B$ with
$0<\mu (\mathbb{R}^d)<\infty $.
Then
\begin{lemma}[Frostman's Lemma]\label{u70}
Let $B\subset \mathbb{R}^d$ be a Borel set. Then $\mathcal{H}^r(B)>0$ if and only if
there exists a $\mu \in \mathcal{M}(B)$ such that
\begin{equation}\label{u73}
 \forall  x\in \mathbb{R}^d, \forall \rho >0,\quad \mu (B(x,\rho ))\leq \rho ^s.
\end{equation}
\end{lemma}
The other ingredient of the proof is the following very well known lemma \cite{Falconer2003}
\begin{lemma}[Mass distribution principle]
Let $B\subset \mathbb{R}^d$. Assume that there exists a measure $\mu \in \mathcal{M}(B)$
and $\delta >0$ such that $\mu (A)<\mathrm{const}\cdot |A|^s$. Then $\dim_{\rm H}(A)\geq s$.
\end{lemma}

\begin{proof}[Proof of Theorem \ref{u1}]
In what follows we always condition on $E\ne\emptyset $. Then by (\ref{155})
\begin{equation}\label{u72}
  \dim_{\rm H}(E)=\frac{\log (p\cdot M^2)}{\log M}=:s, \mbox{ almost surely. }
\end{equation}
It is enough to verify  that
\begin{equation}\label{u71}
\forall q<s,\forall \theta  \in \left(0,\frac{\pi }{4}\right),
\forall \alpha\in \left(\theta  ,\pi/2 -\theta  \right),\quad \dim_{\rm H}(\mathrm{proj}^\alpha (E))>q.
\end{equation}
To see this, we fix an $r$ with $q<r<s$. Then  $\mathcal{H}^r(E)=\infty $.
So, by  Frostman's Lemma there exists a random  measure $\mu \in \mathcal{M}(E)$ such that (\ref{u73})
holds. In particular
\begin{equation}\label{u74}
\forall \mathbf{x}\in \mathcal{E}_n, \quad\mu \left(K_n(\mathbf{x})\right)\leq M^{-nr}.
\end{equation}
Put $\nu _\alpha :=\mathrm{proj}^\alpha _*\mu $.
 Fix an arbitrary  $0<\varepsilon $ and fix an arbitrary
 $0<\rho  $ which is so small that
 \begin{itemize}
   \item for $M^{-(n+1)}<\rho \leq M^{-n}$ we have $n\geq N$,
 for the $N$ defined in (\ref{v6}) and
   \item $nM ^{-nr}<M^{-nq}$.
 \end{itemize}
 Then using these two properties and (\ref{u74}) and \ref{v6} implies that
 $$
 \nu _\alpha (x-\rho, x+\rho )\leq
 10C_2 nM^{-nr}\leq 10C_2 M^{-nq}\leq 10C_2M^q \rho ^{-q}.
 $$
 This completes the proof of (\ref{u71}) by the Mass distribution principle.

\end{proof}

%%ezt javítani és vissza állítani

In the rest of the section we prove
that we can  find slices of angle $\frac{\pi }{4}$ which intersect constant times $n$ level $n$
squares almost surely conditioned on $E\ne\emptyset $.

\begin{proposition}\label{u76}
There exists a constant $0<\lambda <1$ such that
for almost all realizations, conditioned on $E\ne\emptyset $, there exists an $N_6$ such that for all
$n>N_6$ there exists an $\ell\in \mathcal{L}$ with
\begin{equation}\label{u77}
  \#\mathcal{E}_n\left(\ell\right)>\lambda n.
\end{equation}

\end{proposition}

For the proof we need  some new notation and  an easy Fact.

 \medskip
 Let  $D_k$ be the event that all the $M^k$ level-$k$ squares of the diagonal of $[0,1]^2$ is retained. That is
 $$D_k:=
 \left\{\forall \underline{\ell}_k=(\ell_1,\dots ,\ell_k)\in \left\{1,\dots ,M\right\}^k,
 (\underline{\ell}_k,\underline{\ell}_k)\in \mathcal{E}_k
 \right\}.
 $$
 We get
 the definition of the event $D_k^{\underline{i}_n,\underline{j}_n}$ if we substitute the diagonal of $[0,1]^2$ above with the diagonal of $K_{\underline{i}_n,\underline{j}_n}$.
 $$
 D_k^{\underline{i}_n,\underline{j}_n}:=
 \left\{
 \forall \underline{\ell}_k=(\ell_1,\dots ,\ell_k)\in \left\{1,\dots ,M\right\}^k,
( \underline{i}_n\underline{\ell}_k,\underline{j}_n\underline{\ell}_k)\in \mathcal{E}_{n+k}
 \right\}
 $$
A simple argument shows that
\begin{equation}\label{u78}
 p^{2M^k}< \mathbb{P}\left(D_k\right)<p^{M^k}.
\end{equation}
 Let $\Omega '\subset \Omega $ be the set of realizations for which (\ref{155}) holds and let $\mathbb{P}'(\cdot ):=\mathbb{P}\left(\cdot |\Omega '\right)$.

\begin{proof}[Proof of Proposition \ref{u76}]
 Since $pM^2>1$ we can find a
 $\tau $ satisfying $0<\tau <\frac{\log M^2p}{\log 1/p}$. Then
$
  p^{1+\tau }M^2>1
$.
 Therefore we can choose a $0<\gamma <1$ such that
\begin{equation}\label{u75}
p^{\gamma +\tau }M^{2\gamma }>1.
\end{equation}
By (\ref{155})
for all $\omega \in \Omega '$ realization we can find an $N_7=N_7(\omega )$ such that
\begin{equation}\label{u79}
\forall n\geq N_7,\quad  \#\mathcal{E}_n>\left(pM^2\right)^{n\gamma }.
\end{equation}
For every $k$ the events $\left\{D_{k}^{\underline{i}_n,\underline{j}_n}\right\}_{(\underline{i}_n,\underline{j}_n)\in \mathcal{E}_n}$ are independent and each has probability greater than $p^{2M^k}$.
Let $A_n$ be the event that at least one of the events $\left\{D_{k}^{\underline{i}_n,\underline{j}_n}\right\}_{(\underline{i}_n,\underline{j}_n)\in \mathcal{E}_n}$ holds and $A_{n}^{c}$ that non-of them holds. Then
$$
\forall n\geq N_7, \quad
\mathbb{P}'\left(A_{n}^{c}\right)\leq
\left(1-p^{2M^k}\right)^{(pM^2)^{n\gamma }}.
$$
We choose $k=k(n)$ such that
\begin{equation}\label{u80}
  2M^k\leq \tau n<2M^{k+1}.
\end{equation}
Let $a_n:=\left(1-p^{2M^k}\right)^{(pM^2)^{n\gamma }}$
Then $\log a_n<-\left(p^{\gamma +\tau }M^{2\gamma }\right)^n$ which tends to $-\infty $
exponentially fast by (\ref{u75}). Hence
 the series $\sum\limits_{n}\mathbb{P}'(A_{n}^{c})$ is summable. So, Borel Cantelli
 Lemma yields that there exists an $N_6>N_7$ such that for all $n>N_6$ the event
 $A_n$ holds.

 This shows that for $\mathbb{P}'$ almost all realizations $\omega $ for all $m=n+k$ big enough there is an $\ell\in \mathcal{L} $ such that
 $\#\mathcal{E}_m\left(\ell\right)\geq \frac{\tau n}{2M}$. Since $m<2n$ this completes the proof of the proposition.
\end{proof}

\section{Sums of random Cantor sets}

\subsection{Product of percolations} \label{sec:prod}

In this section we consider $d$ independent fractal percolations $\overline{E}^{(1)},\dots ,\overline{E}^{(d)}$ on the line
 with possibly different probabilities $p_1,\ldots,p_d$ but with the same scale $M$.
 The object of this Section is as follows:
 We fix an $\mathbf{a}:=(a_1,\dots ,a_d)\in \mathbb{R}^d$ with $a_i\ne 0$ for all $i=1,\dots ,d$ and consider
  the algebraic sum of coefficients $a_1,\dots ,a_d$:
  $$
 \pmb{ \widetilde{E}_{\mathbf{a}}}:=a_1\cdot \overline{E}^{(1)}+\cdots +a_d\cdot \overline{E}^{(d)}=
 \left\{\sum\limits_{k=1}^{d}
 a_i\cdot e^{(i)}:e^{(i)}\in \overline{E}^{(i)}
 \right\}
 .
  $$
and ask whether such sum contains an interval. It is a generalization of a question solved (in higher generality) by Dekking and Simon in \cite{Dekking2008} for sums of two independent percolations.

Without loss of generality in the rest of the paper we may assume that
\begin{equation}\label{v8}
  p_i> M^{-1},\ \forall i=1,\dots ,d.
\end{equation}
(otherwise the corresponding $\overline{E}^{(i)}$ would be almost surely empty).

%%%%%%%%%%%%%%%%%%%%

%%%%%%%%%%%%%%%%%%%%%%%%

The main result of this Section is as follows:

\begin{theorem} \label{thm:prod}
Assume that
\begin{equation} \label{cond1}
\prod_{i=1}^{d} p_i > M^{-d+1}.
\end{equation}
Then for every $\mathbf{a}=(a_1,\dots ,a_d)\in \mathbb{R}^d$, $a_i\ne 0$ for all $i=1,\dots ,d$
 the sum $\widetilde{E}_\mathbf{a}=\sum\limits_{i=1}^{d} a_i \overline{E}^{(i)}$ contains an interval almost surely, conditioned on all $\overline{E}^{(i)}$ being nonempty.
\end{theorem}
Fix an arbitrary $\mathbf{a}=(a_1,\dots ,a_d)\in \mathbb{R}^d$, $a_i\ne 0$ for all $i=1,\dots ,d$. Without loss of generality we may assume that
\begin{equation}\label{v9}
  \|\mathbf{a}\|=1\mbox{ and } a_i> 0\mbox{ for all } i=1,\dots ,d.
  \end{equation}
Clearly, $\widetilde{E}_\mathbf{a}$ is the orthogonal projection of the random set
$$
\pmb{\widetilde{E}}:=\overline{E}^{(1)}\times \cdots \times \overline{E}^{(d)}\subset [0,1]^d
$$
to the line $\left\{t\cdot \mathbf{a}|t\in \mathbb{R}\right\}$. Hence it follows from  (\ref{u72})
that whenever condition (\ref{cond1}) does not hold then we have
$$
\dim_{\rm H}\widetilde{E}_\mathbf{a}\leq
\dim_{\rm H}\widetilde{E}=\sum\limits_{i=1}^{d}\dim_{\rm H}E^{(i)}
= \sum\limits_{i=1}^{d}\frac{\log Mp_i}{\log M}\leq 1.
$$
\begin{remark}
It is well known that for almost every realization $E$ of a fractal percolation in $\mathbb{R}^d$  $\dim_{\rm H}{E}=1$ implies that $\mathcal{H}^1({E})=0$, see \cite{Mauldin1986}. The same proof goes through for cartesian products of fractal percolations: for typical $\widetilde{E}$  if $\dim_{\rm H}\widetilde{E}=1$ then $\mathcal{H}^1(\widetilde{E})=0$. Hence,  Theorem
\ref{thm:prod} is sharp.

\end{remark}
%Since the fact that $\dim_{\rm H}\widetilde{E}=1$ implies that $\mathcal{H}^1(\widetilde{E})=0$ can be proved in

%We will prove \textbf{\textcolor[rgb]{1.00,0.00,0.00}{later}} that  , so Theorem
%\ref{thm:prod} is sharp.
\subsubsection{Connection between $\widetilde{E}$ and the $d$-dimensional fractal percolation}
Set $
p=\prod_{i=1}^{d} p_i
$  and we write $E$ for the $d$-dimensional fractal percolation with parameters $M,p$. Let $\overline{\mathcal{N}}_n$, $\overline{K}_n(\mathbf{x})$, $\overline{\mathcal{E}}_{n}^{(i)}$ and $\overline{E}^{(i)}_n$ be the one dimensional analogues
of $\mathcal{N}_n, K_n(\mathbf{x})$, $\mathcal{E}_n$ and $E_n$ respectively. That is
$$
\pmb{\overline{\mathcal{N}}_n}:=
\left\{x:x=\left(k+\frac{1}{2}\right)M^{-n},\ k\in \left\{0,1,\dots ,M^n-1\right\}\right\}.
$$
We denote
the level-$n$ interval with center $x\in \overline{\mathcal{N}}_n$ by $\overline{K}_n(x)$.
$$
\pmb{\overline{K}_n(\mathbf{x})}=
x+\left[-\frac{1}{2M^n},\frac{1}{2M^n}\right].
$$
Set
$$
\pmb{\overline{\mathcal{E}}_{n}^{(i)}}:=\left\{x\in \overline{\mathcal{N}}_n:
\overline{K}_n(\mathbf{x}) \mbox{ is retained in the construction of  }E_{n}^{(i)}
\right\}
$$
Finally, the $n$-th approximation of $E^{(i)}$ is denoted by $\overline{E}^{(i)}_n$.
$$
\pmb{\overline{E}^{(i)}_n}:=\bigcup _{x\in \overline{\mathcal{E}}_{n}^{(i)}}
\overline{K}_n(x) \mbox{ and }
\pmb{\widetilde{E}_n}:=\overline{E}^{(1)}_n\times \cdots \times \overline{E}^{(d)}_n.
$$
Then
\begin{equation}\label{v10}
  \widetilde{E}=\bigcap _{n=1}^{\infty }\widetilde{E}_n.
\end{equation}
For an $\mathbf{x}=(x_1,\dots x_d)\in \mathcal{N}_n$ we have
\begin{equation}\label{v12}
 K_n(\mathbf{x})\subset \widetilde{E}_n \Longleftrightarrow
\overline{K}_n(x_i)\subset \overline{E}_{n}^{(i)},\  \forall i\in \left\{1,\dots ,d\right\}.
\end{equation}
Hence for every $\mathbf{x}\in \mathcal{N}_n$ the events that
 $K_n(\mathbf{x})\subset \widetilde{E}_n$ and
$K_n(\mathbf{x})\subset E_n$ share the same probability of $p^n$.
Furthermore, when $\mathbf{x}\in \mathcal{N}_n$ and $\mathbf{y}\in \mathcal{N}_{n-1}$ such that $K_n(\mathbf{x})\subset K_{n-1}(\mathbf{y})$ then
$$
\mathbb{P}\left(\mathbf{x}\in \mathcal{E}_n|\mathbf{y}\in \mathcal{E}_{n-1}\right)=
\mathbb{P}\left(\mathbf{x}\in \widetilde{\mathcal{E}}_n|\mathbf{y}\in \widetilde{\mathcal{E}}_{n-1}\right)=p.
$$

The difference between $E$ and $\widetilde{E}$ follows from the  obvious fact:
\begin{fact}\label{v13} For all distinct $\mathbf{x},\mathbf{y}\in \mathcal{N}_n$  the events:
\begin{description}
  \item[(a)]
$K_n(\mathbf{x})\cap E$ and $K_n(\mathbf{y})\cap E$ are always independent.
  \item[(b)] $K_n(\mathbf{x})\cap \widetilde{E}$ and $K_n(\mathbf{y})\cap \widetilde{E}$ are  independent if and only if $x_i\ne y_i$ for all $i=1,\dots ,d$.
\end{description}
\end{fact}

\subsection{Sections of codimension 1}

We consider hyperplanes
\[
H_t = \{\mathbf{y}\in \mathbb{R}^d: \mathbf{a}\cdot \mathbf{y}= t\}.
\]
That is $H_t(\mathbf{a})$ is the set of points $\mathbf{y}\in \mathbb{R}^d$ whose orthogonal projection to the line with direction vector $\mathbf{a}$ is equal to $t\cdot \mathbf{a}$ (since we assumed that $\mathbf{a}$ is a unit vector).

%%%%%%%%%%%%%%%%%%

\begin{lemma} \label{lem:pi}
Given $p_1,\ldots,p_d>M^{-1}$ satisfying

\[
\prod_i p_i > M^{-d+1},
\]
we can find $q_1,\ldots,q_d$ such that

\[
\prod_i q_i > M^{-d+1},
\]

\begin{equation} \label{eqn:pi}
\prod_{i\neq j} q_i < M^{-d+2} \forall j,
\end{equation}
and

\[
M^{-1} < q_i \leq p_i \forall i.
\]
\begin{proof}
Without weakening the assumptions we can assume that $p_1=\min p_i$. Assume that

\[
\prod_{i=2}^d p_i \geq M^{-d+2}
\]
(otherwise we could choose $q_i=p_i$ for all $i$). There are two cases.

If $p_1 > M^{-1+1/d}$, we can choose any $M^{-1+1/d} < \delta <
M^{-1+1/(d-1)}$ and then set $q_i=\min(\delta, p_1)$ for all $i$.

In the opposite case, let $q_1=p_1$ and set (for $i\geq 2$ and $0\leq
t\leq 1$)
\[
q_i(t) = t p_i + (1-t) p_1.
\]

We have

\[
\prod_{i=2}^d q_i(0) < M^{-d+1} p_1^{-1}
\]
and

\[
\prod_{i=2}^d q_i(1) = \prod_{i=2}^d p_i > M^{-d+2}.
\]

Hence, we can find $t_0\in (0,1)$ such that

\[
M^{-d+1} p_1^{-1} < \prod_{i=2}^d q_i(t_0) < M^{-d+2}
\]
and we can just fix $q_i = q_i(t_0)$ for $i\geq 2$.
\end{proof}
\end{lemma}

To prove the assertion of Theorem \ref{thm:prod} for probabilities $\{p_i\}$ it is enough to prove it for $\{q_i\}$ (increasing of probabilities is not going to decrease probability of the algebraic sum of percolation fractals containing an interval). Hence, we might freely assume that \eqref{eqn:pi} is satisfied for $\{p_i\}$.

%%%%%%%%%%%%%%%%%%%%%%

The goal of this subsection is to prove the following Proposition:

\begin{proposition} \label{prop:prod}
Assume that
\begin{equation}\label{164}
  \prod_{i=1}^{d} p_i > M^{-d+1}\mbox{ and }
  \prod_{i\neq j} p_i < M^{-d+2}\quad \forall j.
\end{equation}
Under assumptions of Theorem \ref{thm:prod}, there is a constant $C$ such that for almost every nonempty realization of $\widetilde{E}$ there is $N$ such that for all $n>N$ for every $t$ if the hyperplane $H_t$ intersects cube $K_n(x_1,\ldots,x_d)$ then it intersects at most $Cn^{(1+\varepsilon)(d-2)}$ other cubes $K_n(y_1,\ldots,y_d)\subset E_n$ such that $x_i=y_i$ for some $i$.
\end{proposition}

First we need some   auxiliary lemmas

\subsubsection{Auxiliary lemmas to the proof of Proposition \ref{prop:prod}}

\begin{lemma} \label{graph}
Let $G=(V,E)$ be a graph such that every vertex $v\in V$ has degree not greater than $n$. Then we can write $V=V_1\cup\ldots\cup V_{n+1}$ in such a way that no edge $e\in E$ connects two vertices from the same $V_i$.
\end{lemma}
\begin{proof}
Let $(V_1,E|V_1)$ be a maximal (in $V$) totally disconnected subgraph. That is, let $V_1\subset V$ such that for any $v_1, v_2\in V_1$, $v_1v_2 \notin E$ but if we added to $V_1$ any additional  point, this property would be lost. In particular, it means that any vertex $v\in V\setminus V_1$ is connected to some $v'\in V_1$ (otherwise $(V\cup \{v\},E|V\cup \{v\})$ would be totally disconnected). It implies that in the graph $(V\setminus V_1, E|V\setminus V_1)$ every vertex has degree not greater than $n-1$. The proof proceeds by induction.
\end{proof}

Yet another auxiliary lemma:

\begin{lemma} \label{lem:lip}
There exists $C>0$, depending only on $d$ and $\{a_i\}$, such that the following holds. Let $F$ be a union of some $M$-adic cubes of level $n$. Assume that the $d-1$-dimensional volume of $F\cap H_t$ is not greater than $Z M^{-n(d-1)}$ for all $t$. Then every $H_t$ intersects at most $CZ$ cubes from $F$.
\end{lemma}

\begin{proof}
Assume that the assertion is not true. Then for every $\varepsilon >0$ one can find a set $F$ satisfying the assumptions such that for some $t$ $H_t$ intersects $\varepsilon^{-1} Z$ cubes from $F$. It implies that there are at least $\varepsilon^{-1} Z/2$ cubes in $F$ such that the $d-1$-dimensional volume of the intersection of $H_t$ with each of them is smaller than $2\varepsilon M^{-n(d-1)}$. Let us denote the family of those cubes by $G$.

Consider now the hyperplanes $L_{t+\frac 1 2 M^{-n}\sum l_i}$ and $L_{t-\frac 1 2 M^{-n}\sum l_i}$. Each cube from $G$ has large intersection (with $(d-1)$-dimensional volume at least $c(d) \cdot M^{-n(d-1)}$) with one of those hyperplanes (the intersection with the other might even be empty). Hence, the sum of $(d-1)$-dimensional volumes of intersecting $G$ with $L_{t+\frac 1 2 \sum l_i}$ and $L_{t-\frac 1 2 \sum l_i}$ is at least $\varepsilon^{-1} c(d)/2 Z M^{-n(d-1)}$, a contradiction.
\end{proof}

We need another useful observation:
\begin{lemma} \label{lem:lip2}
The $d-1$-dimensional volume of the intersection $H_t \cap E_n$ is lipschitz as a function of $t$, with the Lipschitz constant at most $c_5 M^n$.
\end{lemma}
\begin{proof}
There are at most $cM^{n(d-1)}$ cubes $H_t$ can intersect and the volume of intersection of $H_t$ with each of them is Lipschitz with constant $cM^{-n(d-2)}$.
\end{proof}

To prove Proposition \ref{prop:prod}, it is enough to prove the following result:

\begin{proposition} \label{prop:forgot}
Assume that

\[
\prod_{i=1}^d p_i < M^{-d+1}.
\]
Then there is a constant $\tilde{C}$ such that for almost every nonempty realization of $\tilde{E}$ such that for every $t$ the hyperplane $H_t$ intersects at most $\tilde{C}n^{(1+\varepsilon)(d-1)}$ cubes $K_n(x_1,\ldots,x_d)\subset E_n$.
\end{proposition}

Indeed, to prove Proposition \ref{prop:prod} we fix some $i\in \{1,\ldots,d\}$ and intersect $E_n$ with $H_t\cap \{x_i={\rm const}\}$. However, $E_n \cap \{x_i={\rm const}\}$ is just $\overline{E}_n^{(1)} \times \ldots \times \overline{E}_n^{(i-1)} \times \overline{E}_n^{(i+1)} \times \ldots \times \overline{E}_n^{(d)}$ and $H_t \cap \{x_i={\rm const}\}$ is some hyperplane $H_{t'} \subset \mathbb{R}^{d-1}$. By the second part of \eqref{164}, the assumptions of Proposition \ref{prop:forgot} (applied to the cartesian product of all $E^{(j)}, j\neq i$) are satisfied and the assertion of Proposition \ref{prop:forgot} gives us the assertion of Proposition \ref{prop:prod} (note that $d$ in Proposition \ref{prop:forgot} corresponds to $d-1$ in Proposition \ref{prop:prod}).

\subsubsection{The proof of the Proposition \ref{prop:forgot}}
\begin{proof}[The proof of Proposition \ref{prop:forgot}]

The proof will be by induction. For $d=1$ the statement is obvious: the line $H_t$ is just a point, hence it can only intersect one square $K_n(x_1)$. Let us assume the assertion is true for $d-1$ and consider situation for $d$. As $\prod_1^n p_i < M^{-d+1}$, for any $i\in \{1,\ldots,d\}$ we have

\[
\prod_{j\neq i} p_j < M^{-d+2}.
\]
Hence, by the induction assumption for any $i\in \{1,\ldots,d\}$ for all $n$ the codimension 2 hypersurface $H_t\cap \{x_i={\rm const}\}$ intersects at most $\tilde{C}_{d-1}n^{(1+\varepsilon)(d-2)}$ cubes from $E_n$.

Thanks to Lemma \ref{lem:lip}, we only need to estimate (for big $n$) the $d-1$-dimensional volume of $H_t\cap E_n$ to be not greater than $C' M^{-n(d-1)} n^{(1+\varepsilon)(d-1)}$ for all $t$. Let us denote this random variable by $M^{-n(d-1)} \cdot g_n(t)$.
As it is a lipschitz function of $t$, it is enough to check that

\[
g_n(t) \leq C' n^{(1+\varepsilon)(d-1)}
\]
for sufficiently big $n$, not for all $t$ but only for some $M^{-nd}$-dense subset $T_n$. We will choose $\{T_n\}$ in such a way that $T_n \subset T_{n+1}$.

Consider $g_{n+1}(t)$ as a random variable, conditioned on $g_n(t)$. For every cube $K_n(x_1,\ldots,x_d)$ intersecting $H_t$, the volume of $H_t \cap K_n(x_1,\ldots,x_d)$ equals the sum of volumes of $H_t \cap K_{n+1}(y_1,\ldots,y_d)$ over all $K_{n+1}(y_1,\ldots,y_d)\subset K_n(x_1,\ldots,x_d)$ and each $K_{n+1}(y_1,\ldots,y_d)$ appears in $E_{n+1}$ with probability $p$. Hence,

\begin{multline} \label{eqn:g}
\mathbb E(\sum_{K_{n+1}(y_1,\ldots,y_d)\subset K_n(x_1,\ldots,x_d)} \vol (H_t \cap E_{n+1} \cap K_{n+1}(y_1,\ldots,y_d)))\\
= p \vol (H_t \cap E_n \cap  K_n(x_1,\ldots,x_d))
\end{multline}

If events happening in different $K_n(x_1,\ldots,x_d)$ were independent (as it is for fractal percolations), we would be able to estimate $g_{n+1}(t)$ like in the proof of Theorem \ref{u36} because the random variables

\begin{equation} \label{eqn:h}
h_n(x_1,\ldots,x_n)(t) = \sum_{K_{n+1}(y_1,\ldots,y_d)\subset K_n(x_1,\ldots,x_d)} \vol (H_t \cap E_{n+1} \cap K_{n+1}(y_1,\ldots,y_d))
\end{equation}
would be independent. That this is not the case in our situation is the main difficulty in the proof.

Given $n\in \mathbb N$ and $t\in T_n$, we will say the event $B(n,t)$ holds if the random variables $h_n(\cdot)(t)$ can be divided into at most $c_6n^{(1+\varepsilon)(d-2)}$ subfamilies $H_n^i(t)$ such that inside each $H_n^i(t)$ all the  $h_n(\cdot)(t)$ are independent. The constant $c_6$ will be chosen in the future. We will fix the choice of partition $\{H_n^i(t)\}$ (say, the first in a lexicographical order) if many are possible.

We denote

\begin{equation} \label{zi}
z_n^i(t) = \vol (H_t \cap \bigcup_{(x_1,\ldots,x_d)\in H_n^i(t)} K_n(x_1,\ldots,x_d))
\end{equation}
and

\begin{equation} \label{zii}
Z_n^i(t) = \sum_{(x_1,\ldots,x_d)\in H_n^i(t)} h_n(x_1,\ldots,x_n)(t)
\end{equation}

We will say $H_n^i(t)$ is {\it large} if

\[
z_n^i(t) > M^{-n(d-1)} n^{1+\varepsilon},
\]
otherwise it is {\it small}.

For any small $H_n^i(t)$ we can write

\[
Z_n^i(t) \leq z_n^i(t)
\]
(as $E_{n+1}\subset E_n$). For any large $H_n^i(t)$ it has at least $n^{1+\varepsilon}$ elements, and we can apply Azuma-Hoeffding inequality (compare Corollary \ref{u53}) to obtain

\begin{equation} \label{event}
\mathbb P(Z_n^i(t) < (1+\varepsilon) p z_n^i(t)) > 1-\gamma^{n^{1+\varepsilon}}
\end{equation}
for some $\gamma <1$. We say that event $C(n,t)$ holds if $Z_n^i(t) < (1+\varepsilon) p z_n^i(t)$ holds for all large $H_n^i(t)$. Our main interest is the event

\[
A(n,t) = C(n,t) \vee  B^c(n,t),
\]
where $B^c(n,t)$ stands for the complement of the event $B(n,t)$
We claim that, almost surely, there are only finitely many $(n,t)$ for which $A(n,t)$ fails (independently of the choice of $c_6$). Indeed, if $B(n,t)$ fails then $A(n,t)$ is automatically true and if $B(n,t)$ holds the number of $H_n^i(t)$ (large or not) is not greater than $c_6n^{(1+\varepsilon)(d-2)}$. Hence, \eqref{event} implies

\[
\sum_n \sum_{t\in T_n} (1-\mathbb P(A(n,t))) \leq \sum_n M^{nd} c_6n^{(1+\varepsilon)(d-2)} \gamma^{n^{1+\varepsilon}} < \infty
\]
and the claim follows.

Our second claim is that, almost surely, for $c_6$ large enough there are only finitely many $(n,t)$ for which $B(n,t)$ fails. This claim follows from the induction assumption.

Consider any $K_n(x_1,\ldots,x_d)$ and $K_n(y_1,\ldots,y_d)$ intersecting $H_t$. If $x_i\neq y_i$ for all $i$ then $h_n(x_1,\ldots,x_d)(t)$ and $h_n(y_1,\ldots,y_d)(t)$ are independent. We need to estimate for any $(x_1,\ldots,x_d)$ the maximal possible number of different $(y_1,\ldots,y_d)$ such that  $h_n(x_1,\ldots,x_d)(t)$ and $h_n(y_1,\ldots,y_d)(t)$ are not independent.

The induction assumption already gives us the estimation $\tilde{C}_{d-1}n^{(1+\varepsilon)(d-2)}$ for the number of cubes $K_n(y_1,\ldots,y_d)\subset E_n$ intersecting $H_t \cap \{y_i=x_i\}$. The cube can intersect both $H_t$ and $\{y_i=x_i\}$ but be disjoint with $H_t \cap \{y_i=x_i\}$. However, such a cube cannot be too far away from $H_t \cap \{y_i=x_i\}$. The following simple geometric argument shows that the estimation we seek is $c_7\tilde{C}_{d-1}n^{(1+\varepsilon)(d-2)}$, with $c_7$ depending on $\{a_j\}$.

As $K_n(x_1,\ldots,x_d)$ and $K_n(y_1,\ldots,y_d)$ intersect $H_t$, we have

\[
\sum a_j x_j, \sum a_j y_j \in (t-\frac 12 M^{-n} \sum a_j, t+\frac 12 M^{-n} \sum a_j).
\]
As $x_i = y_i$, we have

\[
\sum_{j\neq i} a_j x_j \approx \sum_{j\neq i} a_j y_j \approx t - a_i x_i.
\]
More precisely,

\[
\left| \sum_{j\neq i} a_j (x_j - y_j) \right| \leq M^{-n} \sum a_j.
\]

Hence, the number of such $(y_1,\ldots,y_d)$ is at most as big as the number of $d-1$-dimensional cubes

\[
K_n(y_1,\ldots,y_{i-1},y_{i+1},\ldots,y_d) \in E_n^{(1)} \times \ldots \times E_n^{(i-1)} \times E_n^{(i+1)} \times \ldots \times E_n^{(d)}
\]
intersecting one of at most $2(\sum a_j)/a_i$ hyperplanes

\[
\sum_{j\neq i} a_j y_j = \sum_{j\neq i} a_j x_j + k \min_j a_j,
\]
$k$ varying between $-(\sum a_j) /a_i$ and $(\sum a_j)/a_i$. By the induction assumption, this implies that there are at most $\tilde{C}_{d-1}c_7 n^{(1+\varepsilon)(d-2)}$ such $(y_1,\ldots,y_d)$ for

\[
c_7 = 2\frac {\sum a_j} {\min_j a_j}.
\]
Repeating this reasoning for all possible choices of $i$, for each cube $K_n(x_1, \ldots,x_d)\subset E_n$ intersecting $H_t$ there are at most $\tilde{C}_{d-1}dc_7 n^{(1+\varepsilon)(d-2)}$ cubes $K_n(y_1, \ldots,y_d)\subset E_n$ intersecting $H_t$ such that $x_i=y_i$ for some $i\in \{1,\ldots,d\}$.

We can now apply Lemma \ref{graph} to the dependency graph to divide all the events $h_n(x_1,\ldots,x_d)(t)$ into $\tilde{C}_{d-1}dc_7 n^{(1+\varepsilon)(d-2)}+1$ subfamilies of independent events. This ends the proof of the second claim.

From the two claims, the assertion follows easily. Let

\[
g_n = \sup_t g_n(t).
\]
Then, as soon as $n$ is big enough for $B(n,t)$ and $A(n,t)$ (and hence, $C(n,t)$ as well) always to happen, we will have

\begin{equation} \label{ind}
g_n \leq (1+\varepsilon) p M^{d-1} g_{n-1} + c_6n^{(1+\varepsilon)(d-1)} + c_5,
\end{equation}
where the first term comes from large $H_n^i(t)$, the second term comes from small $H_n^i(t)$, and the third term from lipschitz approximation (\eqref{event} only gives us $g_n(t)$ for $t\in T_n$). As $p M^{d-1}<1$, the inductive formula \eqref{ind} implies the assertion.
\end{proof}

\subsection{Proof of Theorem \ref{thm:prod}}

For $d=2$ Theorem \ref{thm:prod} was already proven in \cite{Dekking2008} but only for the case when the angle is $45^{\circ}$ . Our proof, however, is not similar to \cite{Dekking2008}, rather it has the same flavour as the proof of the main result of \cite{Rams}. We will begin the proof by strengthening the assumptions.

We will study the $d-1$-dimensional volume of $H_t\cap E_n$, we will denote this random variable by $M^{-n(d-1)} \cdot g_n(t)$. We will consider $g_{n+1}(t)$ as a random variable depending on $g_n(t)$. We have equation \eqref{eqn:g}. We define $h_n(x_1,\ldots,x_d)(t)$ by \eqref{eqn:h}. We are going to estimate the volume of $H_t\cap E_n$ from below, using Azuma-Hoeffding Theorem. The main difficulty in the proof is the dependence problem, which we will deal with like in the proof of Proposition \ref{prop:prod}.

The dependence problem is nonexisting for $d=2$. Indeed, any nonhorizontal and nonvertical line $\ell^\alpha$ will intersect only a bounded number of level $n$ squares in any given row (or column). Hence, in this case the argument given in \cite{Rams} works with slight modifications. In what follows, $d\ge 3$.

Given $n,t$, we divide random variables $h_n(\cdot)(t)$ into subfamilies $H_n^i(t)$ such that inside each $H_n^i(t)$ all the events $h_n(\cdot)(t)$ are independent. Like in the proof of Proposition \ref{prop:prod}, we say that the event $B(n,t)$ holds if we can have $c_8 n^{(1+\varepsilon)(d-3)}$ or less families $H_n^i(t)$.
We denote $z_n^i(t)$ and $Z_n^i(t)$ as in \eqref{zi},\eqref{zii}. We will say that $H_n^i(t)$ is {\it large} if

\[
z_n^i(t) > M^{-n(d-1)} n^{1+\varepsilon},
\]
otherwise it is {\it small}.

Conditioned on $E$ being nonempty, almost surely the $d$-dimensional volume of $E_n$ is $\approx cp^n$. Hence, almost surely we will be able to find infinitely many $N_j>N$ and corresponding $t_j$ such that

\[
g_{N_j}(t_j) > p^{N_j} M^{N_j(d-1)(1-\varepsilon)} > e^{\varepsilon N_j} + c_7.
\]
Without weakening the assumptions, $N_j>j$. For $n>N_j$ let $T_n^{(j)}$ be a $M^{-nd}$-dense subset of $I_j = (t_j-M^{-N_j d}, t_j+M^{-N_j d})$ satisfying $T_{n+1}^{(j)} \supset T_n^{(j)}$. In particular,

\[
g_{N_j}(t) >  e^{\varepsilon N_j}
\]
for all $t\in I_j$. For $t\in I_j$ for every small $H_n^i(t)$ we can write

\[
Z_n^i(t) \geq 0.
\]
For large $H_n^i(t)$ the Azuma-Hoeffding inequality gives

\begin{equation} \label{event2}
\mathbb P(Z_n^i(t) > (1-\varepsilon) p z_n^i(t)) > 1-\gamma^{n^{1+\varepsilon}}
\end{equation}
for some $\gamma <1$. We say the event $D(n,t)$ holds if $Z_n^i(t) > (1-\varepsilon) p z_n^i(t)$ for all large $H_n^i(t)$. We define

\[
E(n,t) = D(n,t) \vee  B^c(n,t).
\]
As

\[
\sup_j \sum_{n\geq N_j} \gamma^{n^{1+\varepsilon}} \cdot \sharp T_n^{(j)} < \infty,
\]
almost surely there exist infinitely many $j$'s for which the events $E(n,t)$ hold for all $n\geq N_j$ and $t\in T_n^{(j)}$. Because \eqref{eqn:pi} holds, we can apply Proposition \ref{prop:prod} to prove that, almost surely, events $B(n,t)$ hold for all sufficiently big $n$ for all $t\in T_n^{(j)}$ for all $N_j\leq n$ (like in the proof of the second claim in the main proof of Proposition \ref{prop:prod}). Hence, we can choose $j$ with arbitrarily big $N_j$ such that both $B(n,t)$ and $D(n,t)$ hold for all $n\ge N_j$ for all $t\in T_n^{(j)}$.

We denote

\[
g_n= \inf_{I_j} g_n(t).
\]
We have

\[
g_{n+1} \geq (1-\varepsilon) pM^{d-1} \left( g_n - c_8 n^{(1+\varepsilon)(d-3)} n^{1+\varepsilon}\right) - c_5,
\]
where the first term comes from the growth of the part of $g_n(t)$ contained in the big $H_n^i(t)$ and the second part is the lipschitz correction (the first part we only know for $t\in T_n^{(j)}$). If $N_j$ was big enough, we can prove inductively that

\[
g_n > e^{\varepsilon n} >0
\]
for all $n\geq N_j$. In particular, the algebraic sum of sets $E^{(i)}$ will contain $I_j$. We are done.

\begin{remark}
In the proofs of Proposition \ref{prop:prod} and Theorem \ref{thm:prod} we do not assume that $H_t$ are hyperplanes. They might be any codimension 1 surface sufficiently close to a hyperplane as for the lipschitz property (Lemma \ref{lem:lip2}) to hold. For example, the same argument can be used to show that the assertion of Theorem \ref{thm:prod} holds if we replace algebraic sum with the algebraic multiplication.
\end{remark}

\section{Distance sets for fractal percolations}

In this section we are going to present a related result on distance sets. A long standing conjecture due to Falconer \cite{Falconer1985} says that for any set in $\mathbb{R}^d$ with Hausdorff dimension greater than $d/2$, the distance set has positive length. We prove that for fractal percolation it is enough to require that the Hausdorff dimension is greater than $1/2$ for the distance set to contain an interval.

The proof is almost identical as the proof of Theorem \ref{thm:prod}, so we are only going to sketch it.

For a pair of sets $A,B\in \mathbb R^d$ we define their {\it distance set} as

\[
D(A,B) = \{|x-y|; x\in A, y\in B\}.
\]

Similarly,

\[
D(A) := D(A,A).
\]

\begin{theorem} \label{thm:dist}
Let $E_1, E_2$ be nonempty realizations of two fractal percolations in $\mathbb R^d$ with common scale $M$ and with probabilities $p_1, p_2$. Assume $p_1, p_2 > M^{-d}$ and

\[
p_1 p_2 > M^{-2d+1}.
\]
Then, almost surely $D(E_1, E_2)$ contains an interval.
\end{theorem}

\begin{theorem} \label{thm:dist2}
Let $E$ be a nonempty realization of a fractal percolation in $\mathbb R^d$ for probability $p>M^{-d+1/2}$. Then, almost surely $D(E)$ contains an interval.
\end{theorem}

\begin{proof}
Both theorems are proven in basically the same way. For Theorem \ref{thm:dist} almost surely we can find two cubes: $K_n(x_1,\ldots,x_d)$ with nonempty intersection with $E_1$ and $K_n(y_1,\ldots,y_d)$ with nonempty intersection with $E_2$. For Theorem \ref{thm:dist2} we find two distinct cubes with nonempty intersection with $E$. By going to subcubes, we can freely assume that $x_i\neq y_i$ for all $i$ and that the two cubes are in large distance relative to their size. We can then consider the cartesian product $(E_1 \cap K_n(x_1,\ldots, x_d)) \times (E_2 \cap K_n(y_1,\ldots,y_d)$ (or $(E \cap K_n(x_1,\ldots, x_d)) \times (E \cap K_n(y_1,\ldots,y_d)$) as product of two independent random constructions.

This product is similar to one constructed in section \ref{sec:prod}, but it has fewer dependencies. We can consider its intersections with surfaces

\[
H_t = \{(x,y); \rho(x,y)=t\}.
\]
Those surfaces are sufficiently close to hyperplanes that Lemma \ref{lem:lip2} still holds, though maybe with different constant. The proof of Theorems \ref{thm:dist}  and \ref{thm:dist2}, now reduces to the proof of Theorem \ref{thm:prod}.
\end{proof}

\bibliographystyle{plain}

\bibliography{biblo_5}

\end{document}